\documentclass[10pt,a4paper,reqno]{amsart}
\usepackage{amsmath}
\usepackage{amsthm}
\usepackage{amsfonts}
\usepackage{amssymb}
\usepackage{graphicx}
\usepackage{amsbsy}
\usepackage{mathrsfs}
\usepackage{esint}
\usepackage{bbm}
\usepackage{dsfont}
\usepackage{eucal}
\usepackage{bm}
\usepackage{cite}
\newcommand{\be}{\begin{equation}}
\newcommand{\ee}{\end{equation}}
\newcommand{\ba}{\begin{array}}
\newcommand{\ea}{\end{array}}
\newcommand{\bea}{\begin{eqnarray}}
\newcommand{\eea}{\end{eqnarray}}
\newcommand{\bee}{\begin{eqnarray*}}
\newcommand{\eee}{\end{eqnarray*}}
\newcommand{\tr}{{\rm Tr}}

\newcommand{\ccA}{\mathscr{A}}

\newtheorem{theorem}{Theorem}
\newtheorem{lemma}{Lemma}

\newtheorem{remark}{Remark}
\newtheorem{conjecture}{Conjecture}

\newtheorem*{remark*}{Remark}

\numberwithin{equation}{section}

\catcode`@=11
\def\section{\@startsection{section}{1}%
  \z@{1.5\linespacing\@plus\linespacing}{.5\linespacing}%
  {\normalfont\bfseries\large\centering}}
\catcode`@=12

\newcommand{\R}{\mathbb{R}}

\renewcommand{\leq}{\leqslant}
\renewcommand{\geq}{\geqslant}

\newcommand{\weakto}{\rightharpoonup}
\newcommand{\wto}{\weakto}
\newcommand{\ii}{\infty}
\newcommand\1{{\ensuremath {\mathds 1} }}

\newcommand{\cS}{\mathcal S}
\newcommand{\cE}{\mathcal E}

\renewcommand{\epsilon}{\varepsilon}
\newcommand\pscal[1]{{\ensuremath{\left\langle #1 \right\rangle}}}
\newcommand{\norm}[1]{ \left| \! \left| #1 \right| \! \right| }
\renewcommand{\phi}{\varphi}

\begin{document}

\title{Dynamical Ionization Bounds for Atoms}

\author[E. Lenzmann]{Enno LENZMANN}
\address{Universit\"at Basel, Mathematisches Institut, Rheinsprung 21, CH-4041 Basel, Switzerland.}
\email{enno.lenzmann@unibas.ch}

\author[M. Lewin]{Mathieu LEWIN}
\address{CNRS \& Laboratoire de Math\'ematiques (UMR 8088), Universit\'e de Cergy-Pontoise, F-95000 Cergy-Pontoise, France.}
\email{mathieu.lewin@math.cnrs.fr}

\date{\today}

\begin{abstract} 
We study the long-time behavior of the 3-dimensional repulsive nonlinear Hartree equation with an external attractive Coulomb potential $-Z/|x|$, which is a nonlinear model for the quantum dynamics of an atom. We show that, after a sufficiently long time, the average number of electrons in any finite ball is always smaller than $4Z$ (respectively $2Z$ in the radial case). This is a time-dependent generalization of a celebrated result by E.\,H.~Lieb on the maximum negative ionization of atoms in the stationary case. Our proof involves a novel positive commutator argument (based on the cubic weight $|x|^3$) and our findings are reminiscent of the RAGE theorem. 

In addition, we prove a similar universal bound on the local kinetic energy. In particular, our main result means that, in a weak sense, any solution is attracted to a bounded set in the energy space, whatever the size of the initial datum. Moreover, we extend our main result to Hartree--Fock theory and to the linear many-body Schr\"odinger equation for atoms.

\bigskip

\noindent{\scriptsize\copyright~2012 by the authors. This paper may be reproduced, in its entirety, for non-commercial purposes.}
\end{abstract}

\maketitle

\section{Introduction and Main Result}

The rigorous attempt to answering the question {\em ``How many electrons can a nucleus bind?''}~has received substantial attention in the literature over the last decades~\cite{Ruskai-82a,Ruskai-82,Sigal-82,Sigal-84,Lieb-84,LieSigSimThi-88,Solovej-91,Solovej-03,Nam-12}. So far, the question was only addressed in a time-independent setting, that is, the absence of bound states was shown when the number of electrons in the atom is too large. In the present paper we shall rigorously formulate and provide an answer to a similar question in the time-dependent setting:~{\em `` How many electrons can a nucleus keep in its neighborhood for a long time?''}.

Our main purpose is therefore the rigorous understanding of the long-time behavior of atoms. We shall prove for instance that, in the Hartree approximation, a nucleus of charge $Z$ cannot bind in a time-averaged sense more than $4Z$ electrons (respectively $2Z$ in the radial case). In particular, we will recover some of the known time-independent results (non-existence of bound states) by different arguments, though. One key ingredient in our paper turns out to be a new commutator estimate leading to a novel monotonicity formula, which may be of independent interest for both linear and nonlinear Schr\"odinger equations.

As a model for the quantum dynamics of an atom, let us first consider  the time-dependent nonlinear Hartree equation with an external Coulomb potential:
\begin{equation}
\begin{cases}
\displaystyle i\frac{\partial}{\partial t}u(t,x)=\left(-\Delta-\frac{Z}{|x|}+|u|^2\ast\frac{1}{|x|}\right)u(t,x),\\[0.3cm]
u(0,x)=u_0(x)\in H^1(\R^3).
\end{cases}
\label{eq:Hartree-equation}
\end{equation}
Here $u(t,x)$ describes the quantum state of the electrons (which are treated as bosons for simplicity) in an atom~\cite{Hartree-28,Hartree-28b,Slater-30}. The terms in the parenthesis are respectively the kinetic energy operator of the electrons, the electrostatic attractive interaction with the nucleus of charge $Z$, and the mutual repulsion between the electrons themselves (in units such that $m=2$ and $\hbar=e=1$). The total number of electrons in the system is a conserved quantity, which is given by
$$\int_{\R^3}|u(t,x)|^2\,dx=\int_{\R^3}|u_0(x)|^2\,dx =:N.$$
In physical applications, the number $N$ is an integer but it is convenient to allow any positive real number here. Note that, in Section \ref{sec:extensions} below, we will also consider the physically more accurate Hartree--Fock model as well as the full many-body Schr\"odinger equation describing atoms. But for the time being, we deal  with the Hartree equation.

The nonlinear equation~\eqref{eq:Hartree-equation} and many variations thereof have been studied extensively in the literature. The existence of a unique strong global-in-time solution to \eqref{eq:Hartree-equation} with an initial datum $u_0\in H^1(\R^3)$ goes back to Chadam and Glassey~\cite{ChaGla-75}. Their argument is based on a fixed point argument combined with the conservation of the Hartree energy defined by
\begin{equation}
\cE_Z(u):=\int_{\R^3}|\nabla u(x)|^2\,dx-Z\int_{\R^3}\frac{|u(x)|^2}{|x|}\,dx+\frac12\int_{\R^3}\int_{\R^3}\frac{|u(x)|^2u(y)|^2}{|x-y|}\,dx\,dy.
\label{eq:Hartree-energy}
\end{equation}
In fact, the global well-posedness result for \eqref{eq:Hartree-equation} can be extended to initial data in $L^2(\R^3)$; see for instance~\cite{HayOza-89,Castella-97}. However, in what follows, we will always assume that $u_0$ lies in the energy space $H^1(\R^3)$ so that its corresponding energy is well-defined.

When $Z\leq0$, the solution $u(t)$ to the Hartree equation \eqref{eq:Hartree-equation} exhibits a purely dispersive behavior, which has been studied by many authors. Here, some works were devoted to the understanding of the dispersive effects for any initial datum~\cite{Glassey-77,DiaFig-81,HayOza-87,Hayashi-88,GasIllMarSch-98,SanSol-04}, whereas several others dealt with the construction of (modified) scattering~\cite{GinVel-80,GinOza-93,HayNauOza-98,HayNau-98,GinVel-00,GinVel-00b,LopSol-00,Wada-01,Nakanishi-02}.

In this paper, we are interested in the physically more relevant case when $Z>0$ holds, which corresponds to having an external attractive long-range potential due to the presence of a positively charged atomic nucleus. The electrons can (and will) now be bound by the nucleus, and the problem to understand the long-time behavior of solutions is much more delicate. For instance, it was already noticed by Chadam and Glassey in~\cite[Thm. 4.1]{ChaGla-75} that the solution $u(t)$ cannot tend to zero in $L^\ii(\R^3)$ as $t\to\ii$ for negative energies $\cE_Z(u_0)<0$, which can occur if $Z> 0$ holds. 

When $Z >0$, there exists nonlinear bound states, which are solutions of Equation~\eqref{eq:Hartree-equation} taking the simple form $u(x)e^{-it\lambda}$, where $u \in H^1(\R^3)$ solves the nonlinear eigenvalue equation
\begin{equation}
\left(-\Delta-\frac{Z}{|x|}+|u|^2\ast\frac{1}{|x|}\right)u=\lambda u.
\label{eq:Hartree-stationary}
\end{equation}
For any fixed $0<N\leq Z$, it is known that the equation~\eqref{eq:Hartree-stationary} has infinitely many solutions such that $\int_{\R^3}|u|^2=N$. Moreover, there is a unique positive solution, which minimizes the Hartree energy~\eqref{eq:Hartree-energy}~\cite{LieSim-77,Bader-78} subject to $N$ fixed, and the other (sign-changing) solutions can be constructed by min-max methods~\cite{Wolkowisky-72,Stuart-73,Lions-81}. The interpretation of the condition $0<N\leq Z$ is that the atom is neutral (if $N=Z$) or positively ionized (if $N < Z$). In this situation, it is not energetically favorable to send a positive fraction of $L^2$-mass $\mu > 0$, say, to spatial infinity, since the remaining charge is $Z-(N-\mu)>0$ positive and thus attractive far away from the origin. A more precise mathematical statement is that the Palais--Smale sequences with a bounded Morse index cannot exhibit a lack of compactness when $N\leq Z$, and this implies the existence of infinitely many critical points~\cite{BerLio-83b,Lions-87,Ghoussoub-93}.

It is known that there are bound states in the case of negative ionization, i.\,e.~when $N>Z$ holds. By~\cite[Thm. 7.19]{Lieb-81b} (see also~\cite{Benguria-79,BenBreLie-81}), there is a minimizer of the Hartree functional for $N$ slightly larger than $Z$. However, it is physically clear that there should not be any bound state when $N$ is too large compared to $Z$, because a given nucleus is not expected to bind too many electrons compared to its nuclear charge. In~\cite{Benguria-79,Lieb-81b,Lieb-84}, it was proved that there exists a universal critical constant $1<\gamma_c<2$ such that the equation~\eqref{eq:Hartree-stationary} has \emph{no solution} for $N> \gamma_c Z$ but has at least one for $N\leq \gamma_c Z$. That $\gamma_c$ is independent of $Z$ follows from a simple scaling argument.

Let us now collect some basic facts about the set of solutions of the time-independent problem \eqref{eq:Hartree-stationary}. For  any $u\in H^1(\R^3)$, the self-adjoint operator
$$-\Delta-\frac{Z}{|x|}+|u|^2\ast\frac{1}{|x|}$$
has no positive eigenvalue, by the Kato--Agmon--Simon theorem~\cite[Thm. XIII.58]{ReeSim4}. This shows  that necessarily $\lambda\leq0$ in Equation~\eqref{eq:Hartree-stationary}. Furthermore, we can derive an upper bound on $\| \nabla u \|_{L^2}$ which only depends on $Z$ as follows. If $u \in H^1(\R^3)$ solves \eqref{eq:Hartree-stationary}, then by taking then the scalar product with $u$ we find that
$$\int_{\R^3}|\nabla u(x)|^2\,dx\leq Z\int_{\R^3}\frac{|u(x)|^2}{|x|}\,dx\leq Z\norm{\nabla u}_{L^2}\norm{u}_{L^2}.$$
Here we have used the inequality
$$\int_{\R^3}\frac{|u(x)|^2}{|x|}\,dx\leq \min_{z\geq0}\left(\frac{z}{2}\int_{\R^3}|u|^2+\frac{1}{2z}\int_{\R^3}|\nabla u|^2\right) =\norm{u}_{L^2(\R^3)}\norm{\nabla u}_{L^2(\R^3)}$$
which follows from the value of the hydrogen ground state energy, $\inf\text{Spec}(-\Delta/2-z|x|^{-1})=-z^2/2$.
We conclude that any solution $u \in H^1(\R^3)$ to \eqref{eq:Hartree-stationary} must satisfy the bound
$$\int_{\R^3}|\nabla u|^2\leq \gamma_c Z^{3}.$$ 
Recalling that $\int_{\R^3} |u|^2 \leq \gamma_c Z$ holds, we conclude that the set of all stationary states 
\begin{equation}
\ccA_Z:=\Big\{u\in H^1(\R^3)\ :\ \text{$u$ solves~\eqref{eq:Hartree-stationary} for some $\lambda\leq0$}\Big\}
\end{equation}
is bounded in $H^1(\R^3)$. Elementary arguments show that $\ccA_Z$ is weakly compact in $H^1(\R^3)$. But we note that the set $\ccA_Z$ is not compact in the strong $H^1$-topology. 

Supported by physical reasoning and rigorous results in linear scattering theory about asymptotic completeness (see Remark~\ref{rmk:linear} below), it is  common belief for infinite-dimensional Hamiltonian systems such as  \eqref{eq:Hartree-equation} that any of its solutions should behave for large times as a superposition of one or several states getting closer to the global attractor $\ccA_Z$, plus a dispersive part. This is what has already been shown for $Z\leq0$, in which case $\ccA_Z=\{0\}$. Not much is known in this direction for nonlinear Schr\"odinger equations~\cite{Tao-07,Tao-08}, and solving this problem (a.\,k.\,a.~soliton resolution) constitutes a major mathematical challenge. For the Hartree equation~\eqref{eq:Hartree-equation} studied in this paper, the situation is even less clear because of possible modified scattering due to the long-range effects of the Coulomb potential. We can, however, formulate a simpler (but weaker) conjecture as follows.

\begin{conjecture}[The global attractor]\label{conj:attractor}
Let $u(t)$ be the unique solution to the Hartree equation~\eqref{eq:Hartree-equation} for some $u_0\in H^1(\R^3)$. Take any sequence of times $t_n\to\ii$ such that $u(t_n)\wto u_*$ weakly in $H^1(\R^3)$. Then $u_*\in\ccA_Z$.
\end{conjecture}

\begin{remark}[The many-body Schr\"odinger case]\label{rmk:linear}
Let us recall that the Hartree equation~\eqref{eq:Hartree-equation} is a nonlinear approximation of the linear many-body Schr\"odinger equation
\begin{equation}
\begin{cases}
\displaystyle i\frac{\partial}{\partial t}\Psi(t)=\left(\sum_{j=1}^N\left(-\Delta_{x_j}-\frac{Z}{|x_j|}\right)+\sum_{1\leq k<\ell\leq N}\frac{1}{|x_k-x_\ell|}\right)\Psi(t),\\[0.3cm]
\Psi(0)=\Psi_0\in H^1\left((\R^3)^N\right)
\end{cases}
\label{eq:many-body}
\end{equation}
On the contrary to the Hartree case where we can allow $N=\int_{\R^3}|u|^2$ to take any positive real value, the number $N$ of electrons must of course be an integer for~\eqref{eq:many-body}. The Hartree equation~\eqref{eq:Hartree-equation} is obtained by constraining the solution $\Psi(t)$ to stay on the manifold of product states of the form $\Psi(t,x_1,...,x_N)=\psi(t,x_1)\times\cdots\times\psi(t,x_N)$ and using the Dirac--Frenkel principle. Then $u(t)=\sqrt{N}\psi(t)$ solves~\eqref{eq:Hartree-equation}. 
Let us remark that~\eqref{eq:many-body} can be rewritten after a simple rescaling as follows
\begin{equation}
\begin{cases}
\displaystyle i\frac{1}{Z^2}\frac{\partial}{\partial t}\Psi(t)=\left(\sum_{j=1}^N\left(-\Delta_{x_j}-\frac{1}{|x_j|}\right)+\frac{1}{Z}\sum_{1\leq k<\ell\leq N}\frac{1}{|x_k-x_\ell|}\right)\Psi(t),\\[0.3cm]
\Psi(0)=\Psi_0\in H^1\left((\R^3)^N\right)
\end{cases}
\label{eq:many-body-scaled}
\end{equation}
Thus the limit of large $N \to \infty$ with $N/Z$ fixed corresponds to the usual mean-field limit. In this regime, Hartree's theory is known to properly describe (bosonic) atoms, both for ground states~\cite{BenLie-83b} 
and in the time-dependent case~\cite{ErdYau-01,BarGolMau-00}. See also \cite{Schlein-08, FroLen-04} for a review on mean-field limits and the Hartree approximation.

The many-body equation~\eqref{eq:many-body} looks complicated, but it has the advantage of being linear. In particular, the RAGE theorem tells us that the only possible non-zero weak limits of $\Psi(t)$ when $t\to\ii$ are bounds states of the Hamiltonian $H(N)$ in the parenthesis~\cite{Ruelle-69,AmrGeo-73,Enss-78,ReeSim3}. This is not a very precise description of the solution for large times because if some particles stay close to the nucleus while other escape to infinity, we will always get $\Psi(t)\wto0$ weakly in $H^1(\R^{3N})$, see~\cite{Lewin-11}. However, asymptotic completeness is known to hold for the linear evolution equation \eqref{eq:many-body}. This exactly says that any solution $\Psi(t)$ is, in an appropriate sense, a superposition of bound states of the operators $H(k)$ with $1\leq k\leq N$ and of scattering states~\cite{Derezinski-93,SigSof-94,HunSig-00b}. Because of the behavior of the underlying many-body system, it is reasonable to believe that the same should be true for the Hartree equation~\eqref{eq:Hartree-equation}. 
\end{remark}

A somewhat weaker property that would follow from Conjecture~\ref{conj:attractor} (at least for~\eqref{eq:limit_mass}) is that for large times, the local mass of any solution has to be smaller than  $\gamma_c Z$. 

\begin{conjecture}[Asymptotic number of electrons and kinetic energy]\label{conj:mass}
Let $u(t)$ be the unique solution to the Hartree equation~\eqref{eq:Hartree-equation} for some $u_0\in H^1(\R^3)$. Then
\begin{equation}
\limsup_{t\to\ii}\int_{|x|\leq r}|u(t,x)|^2\,dx\leq \sup_{u\in\ccA_Z}\int_{\R^3}|u|^2=\gamma_c Z
\label{eq:limit_mass}
\end{equation}
and
\begin{equation}
\limsup_{t\to\ii}\int_{|x|\leq r}|\nabla u(t,x)|^2\,dx\leq \sup_{u\in\ccA_Z}\int_{\R^3}|\nabla u|^2\leq \gamma_c Z^3
\label{eq:limit_kinetic}
\end{equation}
for all $r>0$.
\end{conjecture}
The upper bound $\gamma_c Z^3$ is certainly not optimal here. In physical terms, the conjecture says that whatever the number of electrons we start with (and whatever their kinetic energy), we will always end up with at most $\gamma_c Z$ electrons having a universally bounded total kinetic energy. The other electrons have to scatter because the attraction of the nucleus with positive charge $Z$ is not strong enough to keep all the electrons in its neighborhood. It could be that proving the weaker Conjecture~\ref{conj:mass} is not much easier than proving the stronger Conjecture~\ref{conj:attractor}. We actually have very little information on $\gamma_c$. 

In this paper, we are interested in Conjecture~\ref{conj:mass}. We will prove a time-averaged version of~\eqref{eq:limit_mass}, with $\gamma_c$ replaced by $2$ in the radial case, and by $4$ in the general case. Our main result is as follows.

\begin{theorem}[Long-time behavior of atoms in Hartree theory]\label{thm:Hartree}
Suppose $Z>0$, let $u_0$ be an arbitrary initial datum in $H^1(\R^3)$, and denote by $u(t)$ the unique solution of~\eqref{eq:Hartree-equation}.  Then, for any $R > 0$, we have the following estimate 
\begin{equation}
\frac1{T} \int_0^Tdt\;\int_{\R^3}dx\,\frac{|u(t,x)|^2}{1+|x|^2/R^2}\leq 4Z+\frac{3}R+\frac{2\,\sqrt{KN}R^2}{ZT}
\label{eq:estimate_average}
\end{equation}
with
$$N:=\int_{\R^3}|u_0|^2$$
and
\begin{equation}
K:=\sup_{t\geq0}\int_{\R^3}|\nabla u(t)|^2\leq {Z^2N+ 2\norm{\nabla u_0}^2_{L^2(\R^3)}+N^3\norm{\nabla u_0}_{L^2(\R^3)}}.
\label{eq:unif_K} 
\end{equation}
In particular, we have 
\begin{equation}
{\limsup_{T\to\infty}\frac1{T} \int_0^Tdt\int_{|x|\leq r}dx\,|u(t,x)|^2 \leq 4Z}
\label{eq:limsup}
\end{equation}
for every $r>0$. Similarly, we have the following estimate on the local kinetic energy
\begin{multline}
\frac1{T} \int_0^T\!dt\int_{\R^3}dx\,\frac{|\nabla u(t,x)|^2}{(1+|x|/R)^2}\leq \left(\frac{Z^2}{4}+\frac{2Z}R+\frac{3Z}{R^2}\right)\frac1{T} \int_0^Tdt\;\int_{\R^3}dx\frac{|u(t,x)|^2}{1+|x|^2/R^2}\\+\frac{2R\sqrt{K}\sqrt{N}}{T}
\label{eq:estimate_average_kinetic}
\end{multline}
and therefore
\begin{equation}
{\limsup_{T\to\infty}\frac1{T} \int_0^Tdt\int_{|x|\leq r}dx\, | \nabla u(t,x)|^2 \leq Z^3}
\label{eq:limsup_K}
\end{equation}
for every $r>0$.

If the initial datum $u_0=u_0(|x|)$ is radial, then $u(t)$ is radial for all times and the same estimate~\eqref{eq:estimate_average} holds true with $4Z$ replaced by $2Z$. Similarly, the estimate \eqref{eq:limsup_K} holds true with $Z^3$ replaced by $Z^3/2$. 
\end{theorem}

Note that we do not exactly get that the limiting mass is $\leq 4Z$ for large times, but we only know it in the sense of \emph{time averages} of the form $\langle f \rangle_T = T^{-1}\int_0^T f\,dt$. Such a statement is reminiscent of the celebrated RAGE theorem~\cite{Ruelle-69,AmrGeo-73,Enss-78,ReeSim3} for linear time evolutions generated by self-adjoint operators. The constants in the error terms of~\eqref{eq:estimate_average} and~\eqref{eq:estimate_average_kinetic} are probably not optimal at all, but they are displayed here to emphasize that our method can provide simple and explicit bounds. However, we have not tried to optimize these constants too much.

In the radial case, we are able to get the same numerical value of $2$ as the best known estimate on $\gamma_c$. However, we use a virial-type argument that seems to be quite different from Lieb's celebrated proof in~\cite{Lieb-84} in the stationary case (which, for radial solutions, goes back to Benguria~\cite{Benguria-79}). In particular, our approach provides an alternative proof of the fact that $\gamma_c<2$ in the stationary radial case.

\subsection*{Strategy of the Proof}
Now, we explain the main ideas used in the proof of Theorem~\ref{thm:Hartree}. To this end, we start by quickly recalling Lieb's proof~\cite{Lieb-84} that $\gamma_c<2$ holds. His idea is to take the scalar product of the stationary Hartree equation~\eqref{eq:Hartree-stationary} with $|x|u(x)$, leading to the estimate
$$\pscal{u,\frac{|x|(-\Delta)+(-\Delta)|x|}{2}u}-ZN+\int_{\R^3}\int_{\R^3}\frac{(|x|+|y|)|u(x)|^2|u(y)|^2}{2|x-y|}\, dx\,dy\leq 0,$$
using that $\lambda\leq0$ holds. To conclude, it suffices to notice that we have
$$\frac{|x|(-\Delta)+(-\Delta)|x|}2=|x|^{1/2}\left(-\Delta-\frac{1}{4|x|^2}\right)|x|^{1/2}>0,$$
by Hardy's inequality, and that
$$\frac{|x|+|y|}{|x-y|}\geq 1$$
by the triangle inequality. Combining these estimates, we obtain that $-ZN+N^2/2< 0$, which implies the bound $N< 2Z$ for the stationary problem \eqref{eq:Hartree-stationary}. (Note that the inequality is strict, since there is no optimizer in Hardy's inequality.)

In view of Lieb's argument for the stationary problem \eqref{eq:Hartree-stationary}, it appears to be a viable strategy in the time-dependent setting to consider the quantity $M(t) = \int |x| |u(t,x)|^2 \, dx$ (or some spatially localized version thereof). Indeed, if we take the second time derivative of $M(t)$, we are (formally) led to the well-known {\em Morawetz--Lin--Strauss estimate} for nonlinear Schr\"odinger equations (NLS), which has been proved of enormous value in the setting of NLS with purely repulsive interactions. However, due to presence of the attractive term $-Z/|x|$ with $Z> 0$ in the Hartree equation \eqref{eq:Hartree-equation}, the use of the classical Morawetz--Lin--Strauss bounds does not yield any dispersive information about $u(t,x)$, even in the case when $N$ is large compared to $Z$. 

In our situation, it turns out that it is more natural to study the time evolution of the third moment $M(t) = \int |x|^3 |u(t,x)|^2 \, dx$.
If we compute its second time derivative, we obtain
\begin{align*}
\frac13 \frac{d^2}{dt^2} \int_{\R^3}|x|^3|u(t,x)|^2\,dx&=\frac{d}{dt}\pscal{u(t),Au(t)}\\
&=2\Re\pscal{\frac{\partial}{\partial t}u(t),Au(t)}\\
&=\pscal{u(t),i[-\Delta,A]u(t)}+\pscal{u(t),i[V_u,A]u(t)} 
\end{align*}
with 
$A:=-i\big(\nabla\cdot x\,|x|+|x|\,x\cdot \nabla\big)$
and $V_u=-Z|x|^{-1}+|u|^2\ast|x|^{-1}$. This is the same as multiplying the time-dependent equation~\eqref{eq:Hartree-equation} by $A\overline{u(t)}$ and taking the imaginary part.
Our key observation is the positivity of the commutator
\begin{equation}
i[-\Delta,A]=-\frac13\big[\Delta,\big[\Delta,|x|^3\big]\big]\geq0
\label{eq:idea_positive_comm} 
\end{equation}
(see also \eqref{eq:double-comm-powers} below), combined with the fact that
\begin{align*}
\pscal{u(t),i[V_u,A]u(t)}  & = -2\int_{\R^3}|x|x\cdot\nabla V_u(x) |u(t,x)|^2\\
&=\int_{\R^3}\int_{\R^3}(|x|x-|y|y)\cdot\frac{x-y}{|x-y|^3}|u(t,x)|^2|u(t,y)|^2dx\,dy-2ZN\\
&\geq \kappa\, N^2-2 ZN,
\end{align*}
where $\kappa=1$ if $u(t)$ is radial and $\kappa=1/2$ otherwise (see Lemma~\ref{lem:estim_potential_x_3} below). Hence, when $N>2Z/\kappa$, we deduce the lower bound
\begin{equation} \label{ineq:virial}
\frac13\frac{d^2}{dt^2}\int_{\R^3}|x|^3|u(t,x)|^2\,dx=\frac{d}{dt}\pscal{u(t),Au(t)}\geq N(\kappa N-2Z)>0 .
\end{equation}
Therefore the quantity $\int_{\R^3}|x|^3|u(t,x)|^2\,dx$ grows at least like $t^2$ for large $t$ and in particular $\pscal{u(t),Au(t)}$ is a monotone increasing quantity. This growth is a strong indication that some dispersion takes place and some particles have to escape to infinity. (A regularized version of the previous estimate will indeed show this claim for any $H^1$-solution.) Note also that, in the time-independent case when $u$ is a nonlinear bound state (and hence the left side in \eqref{ineq:virial} must be zero), this is also a new proof of Lieb's inequality $\gamma_c<2$ in the radial setting, since $\kappa =1$ holds under this symmetry assumption.

Let us generally remark that virial or positive commutator arguments are very common in the literature~\cite{KilVis-08,ColKelStaTakTao}. When $|x|^3$ is replaced by $|x|$ this leads to the famous Morawetz inequalities~\cite{Morawetz-68} as already mentioned, whereas the case of $|x|^2$  gives the virial identity used by Glassey in~\cite{Glassey-77b} to prove finite-time blowup for NLS. In a recent work~\cite{Tao-08}, Tao advocated the use of $|x|^4$ for some nonlinear Schr\"odinger equations in dimension $d\geq7$, in order to get a universal bound on the mass of the solution. We are not aware of any use of the multiplier $|x|^3$ in the literature.

In fact, using the cubic weight $|x|^3$ is rather natural from a dimensional point of view in our situation: If the potential term $[V_u,A]$ should be $O(1)$, then the virial function must behave like the third power of a length to compensate the Laplacian and the Coulomb potential. 

For the proof of our main result, we will in fact derive a whole class of double commutator estimates of the same kind as~\eqref{eq:idea_positive_comm}, which we think is of independent interest too. In particular, we will show in~\eqref{eq:double-comm-powers} below that, in any dimension $d\geq1$, we have the commutator bound
\begin{equation}
-\big[\Delta,\big[\Delta,|x|^\beta\big]\big]\geq \beta(\beta+d-4)(d-\beta)|x|^{\beta-4},
\label{eq:general_dble_comm} 
\end{equation}
provided that $\beta\geq \max(1,4-d)$. Note that the right side is $\geq 0$ when $\beta\leq d$. In spite of the fact that~\eqref{eq:general_dble_comm} turns out to be equivalent to a general version of Hardy's inequality, we have not found it explicitly written (let alone systematically treated) in the literature. Notice that the bound~\eqref{eq:general_dble_comm} contains the usual inequalities for $\beta=1,2$, as well as Tao's estimate for $\beta=4$. In the present application, we shall use \eqref{eq:general_dble_comm} in dimension $d=3$ with $\beta=3$, or rather a regularized version thereof. However, let us remark that the positivity of this commutator does not directly follow as in the ``classical'' cases when $\beta = 1,2$. To wit this, we note that, for $d=\beta=3$, a calculation (which will be detailed below) yields the identity
\begin{align*}
-\big [\Delta, \big [\Delta, |x|^3 \big ] \big]  & = -\Delta \Delta |x|^3 - \nabla \cdot ( \mathrm{Hess}_{|x|^3} ) \nabla  \\ & = -\frac{24}{|x|} - 12 \nabla \cdot  \left [ |x| \left ( 1 + \omega_x \omega_x^T \right ) \right ] \nabla,  
\end{align*}
where $\omega_x = \frac{x}{|x|}$ denotes the unit vector in direction $x \in \R^3$. Obviously, the first term on the right side is negative definite. Nevertheless, when combined with the second term, the generalized Hardy's inequality (see \eqref{eq:Hardy-power} below) shows that we indeed have that the whole right-hand side is non negative, and hence the estimate \eqref{eq:general_dble_comm} follows in the particular case $d=\beta=3$. 

Ultimately, we are interested in general $H^1$-solutions $u(t)$ without imposing any spatial weight condition. Therefore, the strategy of proving Theorem \ref{thm:Hartree} explained above needs to be further refined.  In particular, the desired bound~\eqref{eq:estimate_average} on a ball of radius $R$ cannot be obtained by only looking at the second derivative of the third moment as we have just explained. Our method to extend~\eqref{eq:estimate_average} to any $H^1$-valued solution $u(t)$ is to replace the function $|x|^3$ by a radial function $f_R(|x|)$ which behaves like $|x|^3$ on the ball of radius $R$ and like $|x|$ at infinity. This will imply that $A_{f_R} = -i[\Delta, f_R]$ defines a bounded operator on $L^2(\R^3)$. Furthermore, we will need to derive a sufficiently good lower bound on the double commutator $-[\Delta,[\Delta,f_R]]$ in order to imitate the previous argument on the ball only. In Section \ref{sec:commutators}, we explain how to do this for a general function $f$. 
Finally, the bound~\eqref{eq:estimate_average_kinetic} on the local kinetic energy is itself obtained by considering another virial function $g_R$ which behaves like $|x|^2$ on the ball of radius $R$ and like $|x|$ at infinity.
The complete proof of Theorem~\ref{thm:Hartree} is given in Section~\ref{sec:proof}.

\subsection*{Extensions: Hartree-Fock and Many-Body Schr\"odinger Theory}
In physical reality, electrons are fermions, which means that the many-body wave function $\Psi=\Psi(t,x_1,...,x_N)$ in~\eqref{eq:many-body} must be \emph{antisymmetric} with respect to exchanges of its spatial variables $x_1,...,x_N$. The Hartree state $\psi(t,x_1)\cdots\psi(t,x_N)$ is symmetric and it is therefore not allowed for physical electrons. This is why one speaks about \emph{bosonic} atoms. The simplest product-like antisymmetric wave function is a \emph{Hartree-Fock state} sometimes also called a Slater determinant
$$\Psi(t,x_1,...,x_N)=\frac{1}{\sqrt{N!}}\sum_{\sigma\in\cS_N}\varepsilon(\sigma)\,u_1(t,x_{\sigma(1)})\cdots u_N(t,x_{\sigma(N)}),\qquad \pscal{u_j,u_k}_{L^2}=\delta_{jk}.$$
In Section~\ref{sec:Hartree-Fock} below, we extend Theorem~\ref{thm:Hartree} to the corresponding time-dependent Hartree-Fock equations; see Theorem \ref{thm:Hartree-Fock} for a precise statement. Finally, we also consider the full many-body Schr\"odinger equation~\eqref{eq:many-body} in Section~\ref{sec:many-body} below, where our findings are summarized in Theorem \ref{thm:many-body}.

\section{Estimating the Commutator $-[\Delta,[\Delta,f(x)]]$}\label{sec:commutators}

Throughout this section, we use the convenient notation 
$$p :=-i\nabla,$$
and in particular we have $p^2 = -\Delta$ in what follows. In this section, we investigate how to get lower bounds for a double commutator of the form $-[p^2,[p^2,f(x)]]$ in general space dimensions $d \geq 1$.  Such a double commutator always arises when computing the second derivative of the expectation value of $f(x)$, in a non-relativistic system based on the Laplacian. We always assume that $f$ is smooth enough (possibly only outside of the origin), such that the double commutator can be at least properly interpreted as a quadratic form on $C^\ii_c(\R^d)$ or on $C^\ii_c(\R^d\setminus\{0\})$.

Our starting point is the well-known formula for the double commutator, which follows from a tedious but simple calculation:
\begin{equation}
-\big[p^2\,,\,[p^2,f(x)]\,\big]=-(\Delta\Delta f)(x)+4 p\cdot \big({\rm Hess}\,f(x)\big)\,p.
\label{eq:dble_comm}
\end{equation}
Since the Hessian of $f$ appears on the right side, it is natural to restrict to convex functions $f$. Then the second term is non-negative in the sense of operators. 
One can use this term to control the bi-Laplacian of $f$ by resorting to Hardy's trick, which is based on writing
\begin{align}
p\cdot \big({\rm Hess}\,f(x)\big)\,p&=\big(p+iF(x)\big)\cdot \big({\rm Hess}\,f(x)\big)\,\big(p-iF(x)\big)\nonumber\\
&\quad+i\left(p\cdot \big({\rm Hess}\,f(x)\big) F(x)- F(x)\cdot \big({\rm Hess}\,f(x)\big) p\right)\nonumber\\
&\quad -F(x)\cdot \big({\rm Hess}\,f(x)\big) F(x)\nonumber\\
&\geq {\rm div}\left({\rm Hess}\,f(x) F(x)\right) - F(x)\cdot \big({\rm Hess}\,f(x)\big) F(x)
\label{eq:Hardy-trick}
\end{align}
for any sufficiently smooth real vector field $F:\R^3\to\R^3$. Here we have only used that 
$\big(p+iF(x)\big)\cdot \big({\rm Hess}\,f(x)\big)\,\big(p-iF(x)\big)\geq0$ holds, which simply follows from the assumed convexity ${\rm Hess}\,f(x)\geq0$ and the self-adjointness $(p+iF(x))^*=p-iF(x)$. 
For dimensional reasons, it is natural to take $F$ of the form $F(x)=\alpha x|x|^{-2}$ with some constant $\alpha\in\R$. We thus obtain the lower bound
\begin{align}
-\big[p^2\,,\,[p^2,f(x)]\big]&=4\left(p+i\alpha\frac{x}{|x|^2}\right)\cdot \big({\rm Hess}\,f(x)\big)\,\left(p-i\alpha\frac{x}{|x|^2}\right)\nonumber\\
&\quad+4\alpha\, {\rm div}\left({\rm Hess}\,f(x)\frac{x}{|x|^2}\right)
-4\alpha^2\frac{x^T\big({\rm Hess}\,f(x)\big)x}{|x|^4}-(\Delta\Delta f)(x)\nonumber\\
&\geq 4\alpha\, {\rm div}\left({\rm Hess}\,f(x)\frac{x}{|x|^2}\right)
-4\alpha^2\frac{x^T\big({\rm Hess}\,f(x)\big)x}{|x|^4}-(\Delta\Delta f)(x)
\label{eq:general-commutator-estimate}
\end{align}
for a sufficiently smooth convex function $f$ and any $\alpha\in\R$. Note that by using Hardy's trick we are able to obtain a lower bound which does not contain the differential operator $p$. Our estimate only involves a multiplication operator. By varying $\alpha$, we can try to make the negative part of this function as small as possible.

Let us now restrict ourselves to a radial function $f(|x|)$ and use the notation $r=|x|$ and $\omega_x:=x/|x|$ for simplicity. Some tedious calculations show that
$${\rm Hess\ }f(|x|)=(1-\omega_x\omega_x^T)\frac{f'(r)}{r}+\omega_x\omega_x^Tf''(r),$$
\begin{equation*}
{\rm div}\left(\big({\rm Hess}\,f(x)\big)\frac{x}{|x|^2}\right)={\rm div}\left(\frac{f''(r)}{r}\omega_x\right)=\frac{f^{(3)}(r)}{r}+(d-2)\frac{f''(r)}{r^2},
\end{equation*}
\begin{equation*}
\frac{x^T\big({\rm Hess}\,f(x)\big)x}{|x|^4}=\frac{f''(r)}{r^2}.
\end{equation*}
Moreover, we recall the formula for the Bi-Laplacian of a radial function:
$$\Delta\Delta f(|x|)=f^{(4)}(r)+2(d-1)\frac{f^{(3)}(r)}{r}+(d-1)(d-3)\frac{f^{(2)}(r)}{r^2}-(d-1)(d-3)\frac{f'(r)}{r^3}.$$
Therefore we can rewrite the equality in \eqref{eq:general-commutator-estimate} for a radial function $f$ as follows:
\begin{align}
-\big[p^2\,,\,[p^2,f(|x|)]\big]=&4\left(p+i\alpha\frac{\omega_x}{r}\right)\cdot\left((1-\omega_x\omega_x^T)\frac{f'(r)}{r}+\omega_x\omega_x^Tf''(r)\right)\,\left(p-i\alpha\frac{\omega_x}{r}\right)\nonumber\\
& -f^{(4)}(r)+4\left(\alpha-\frac{d-1}{2}\right)\frac{f^{(3)}(r)}{r}\nonumber\\
&+4\left(\alpha(d-2)-\alpha^2-\frac{(d-1)(d-3)}4\right)\frac{f''(r)}{r^2}\nonumber\\
&+(d-1)(d-3)\frac{f'(r)}{r^3}.
\label{eq:radial-commutator-estimate}
\end{align}
The operator on the first line is $\geq0$ when $x\mapsto f(|x|)$ is convex. In dimension $d=3$, we already get a simple estimate.

\begin{lemma}[A lower bound for $d=3$]
Let $f:[0,\ii)\to\R$ be a convex non-decreasing function such that $x\mapsto f^{(4)}(|x|)\in L^1_{\rm loc}(\R^3)$. Then we have
\begin{align}
-\big[p^2\,,\,[p^2,f(|x|)]\big] &= 4\left(p+i\frac{\omega_x}{r}\right)\cdot\left((1-\omega_x\omega_x^T)\frac{f'(r)}{r}+\omega_x\omega_x^Tf''(r)\right)\,\left(p-i\frac{\omega_x}{r}\right)\nonumber\\
&\qquad-f^{(4)}(|x|)\nonumber\\
&\geq -f^{(4)}(|x|)
\label{eq:simple-bound-1D-3D}
\end{align}
in the sense of quadratic forms on $C^\ii_c(\R^3)$.
\end{lemma}

\begin{proof}
Take $\alpha=1$ in~\eqref{eq:radial-commutator-estimate}.
\end{proof}

\medskip

Coming back to~\eqref{eq:radial-commutator-estimate} and taking now the convex function $f(|x|)=|x|^\beta$ with $\beta\geq1$, we obtain the following general result.
\begin{lemma}[Estimate on {$-[p^2,[p^2,|x|^\beta]]$}] \label{lem:lower}
For all $\beta\geq\max(1,4-d)$, we have
\begin{equation}
-[p^2,[p^2,|x|^\beta]] \;\geq\; \beta(\beta+d-4)(d-\beta)\;|x|^{\beta-4},
\label{eq:double-comm-powers}
\end{equation}
in the sense of quadratic forms on $C^\ii_c(\R^d)$ (resp. on $C^\ii_c(\R^d\setminus\{0\})$ if $\beta=4-d$). The right side of~\eqref{eq:double-comm-powers} is non negative for $\max(1,4-d)\leq \beta\leq d$.
\end{lemma}

\begin{proof}
Take $f(r)=r^\beta$ in~\eqref{eq:radial-commutator-estimate} and optimize with respect to $\alpha$ (the optimum is $\alpha=(\beta+d-4)/2$).
We need $\beta\geq1$ to make sure that $f$ is non-decreasing and convex, and $\beta>4-d$ to ensure that all the terms are in $L^1_{\rm loc}(\R^d)$. For $\beta=4-d\geq1$, the right side of~\eqref{eq:radial-commutator-estimate} vanishes and the bound stays correct by a simple limit argument. We remark that, in the borderline case $\beta = 4-d$, there is a positive $\delta$-measure occurring at the origin $x=0$, which we do not see when using functions of $C^\ii_c(\R^d\setminus\{0\})$.
\end{proof}

\begin{remark}
Note also the special formula $-[p^2, [p^2, |x|^2]] = 8 p^2 \geq 0$ valid in any dimension $d \geq 1$, which immediately follows from \eqref{eq:dble_comm}. For $d \geq 3$ and $\beta=2$, the lower bound given in Lemma \ref{lem:lower} is then a direct consequence of Hardy's inequality $4p^2 \geq (d-2)^2 |x|^{-2}$. In fact, we shall see below that the bound in Lemma \ref{lem:lower} is equivalent to a generalized version of Hardy's inequality.  
\end{remark}

We conclude this section with some general observations as follows. First, we note that Lemma \ref{lem:lower} gives a non negative lower bound in~\eqref{eq:double-comm-powers} in dimension $d=2$ for the choice $\beta = 2$ only. In higher dimensions $d\geq3$, the right side is non negative for any $1\leq \beta\leq d$. When $\beta=4$, we get the simple lower bound
$$-[p^2,[p^2,|x|^4]] \;\geq\; 4d(d-4),\qquad \text{for $d\geq4$},$$
which was used for the first time by Tao in~\cite{Tao-08}.

As we have seen, the bound~\eqref{eq:double-comm-powers} is equivalent to the operator inequality
$$\left(p+i\alpha\frac{\omega_x}{r}\right)\cdot\left((1-\omega_x\omega_x^T)\frac{f'(r)}{r}+\omega_x\omega_x^Tf''(r)\right)\,\left(p-i\alpha\frac{\omega_x}{r}\right)\geq0$$
with $f(r)=r^\beta$. This can also be written for the optimal $\alpha=(\beta+d-4)/2$ as
$$\int_{\R^d}|x|^{\beta-2}\left(\left|P^\perp_x\nabla u(x)\right|^2+(\beta-1)\left|\omega_x\cdot\nabla u(x)+\frac{\beta+d-4}{2|x|}u(x)\right|^2\right)dx\geq0,$$
where $P_x^\perp=1-\omega_x\omega_x^T$ is the projection on the two-dimensional space orthogonal to $\omega_x$.
Saying that the second term is non negative is equivalent, for $\beta>1$, to the (generalized) Hardy inequality
\begin{equation}
\int_{\R^d}|x|^{\beta-2}|\omega_x\cdot \nabla u(x)|^2\,dx\geq \frac{(\beta+d-4)^2}{4}\int_{\R^d}|x|^{\beta-4}|u(x)|^2\,d x.
\label{eq:Hardy-power}
\end{equation}
Hence we see that~\eqref{eq:double-comm-powers} is nothing else but a reformulation of Hardy's inequality~\eqref{eq:Hardy-power}.

\begin{remark}[Fractional Laplacians]
Using the integral representation
$$x^\theta=\frac{\sin(\pi \theta)}{\pi}\int_0^\ii \frac{x}{x+s}\,s^{\theta-1}ds, \quad \mbox{for $0 < \theta < 1$}, $$
we can easily transpose most of our estimates to fractional powers $|p|^{2\theta} = (-\Delta)^\theta$ and $\langle p \rangle^{2\theta} = (|p|^2 +1)^{\theta}$ with $\theta \in (0,1)$. For instance, for the pseudo-relativistic kinetic energy operator $\sqrt{p^2+1}$, we have, at least formally,
\begin{multline*}
-\Big[\sqrt{1+p^2},\big[\sqrt{1+p^2},f(x)\big]\Big]=\frac{1}{\pi^2}\int_0^\ii \sqrt{s}\,ds\int_0^\ii \sqrt{t}\,dt\times\\
\times\frac{1}{(1+p^2+s)(1+p^2+t)}\Big(-\big[p^2,[p^2,f(x)]\big]\Big)\frac{1}{(1+p^2+s)(1+p^2+t)}.
\end{multline*}
In particular we find
\begin{equation*}
-\Big[\sqrt{1+p^2},\big[\sqrt{1+p^2},|x|^\beta\big]\Big]\geq0
\end{equation*}
for $\max(1,4-d)\leq\beta\leq d$. For a general convex radial function $f$ and in $d=3$ dimensions, we obtain the estimate
$$-\Big[\sqrt{1+p^2},\big[\sqrt{1+p^2},f(|x|)\big]\Big]\geq -\frac14\norm{f^{(4)}_+}_{L^\ii(\R^3)}$$
with $f^{(4)}_+$ denoting the positive part of $f^{(4)}$.
\end{remark}

\section{Proof of Theorem~\ref{thm:Hartree}}\label{sec:proof}

In this section, we provide the proof of our main result given by Theorem~\ref{thm:Hartree}. We always assume that the initial datum $u_0$ is smooth and decays fast enough, such that our calculations are justified. As we will see below, our estimates only involve the $H^1(\R^3)$ norm of $u_0$, and thus the general case can therefore be obtained by a simple limiting argument, which we do not detail here.

\subsection*{Step 1. The Virial Identity}

Consider a smooth radial convex function $f$. We define the corresponding virial operator
\begin{equation}
A_f:=p\cdot\nabla f+\nabla f\cdot p=p\cdot \omega_x f'(|x|)+f'(|x|)\omega_x\cdot p.
\label{eq:def_A_f} 
\end{equation}
Using Formula~\eqref{eq:simple-bound-1D-3D} of the previous section, we get
\begin{align}
\frac{d}{dt}\pscal{u(t),A_fu(t)}=&4\int_{\R^3}\frac{f'(|x|)}{|x|}|P^\perp_x\nabla u(t,x)|^2\,dx\nonumber\\
&+4\int_{\R^3}f''(|x|)\left|\omega_x\cdot\nabla u(t,x)+\frac{u(t,x)}{|x|}\right|^2\,dx\nonumber\\
&-\int_{\R^3}f^{(4)}(|x|)|u(t,x)|^2\,dx-2\int_{\R^3}f'(|x|)|u(t,x)|^2\omega_x\cdot \nabla V_{u}(t,x)\,dx,
\label{eq:first-Virial}
\end{align}
with
$$V_{u}(t,x)=-\frac{Z}{|x|}+|u(t)|^2\ast|x|^{-1}:=-\frac{Z}{|x|}+W_u(t,x).$$
The first potential term is just
$$-2\int_{\R^3}f'(|x|)|u(t,x)|^2\omega_x\cdot \nabla \left(-\frac{Z}{|x|}\right)\,dx=-2Z\int_{\R^3}\frac{f'(|x|)}{|x|^2}|u(t,x)|^2.$$
The second potential term can be expressed as follows
\begin{align*}
&-2\int_{\R^3}f'(|x|)|u(t,x)|^2\omega_x\cdot \nabla W_{u}(t,x)\,dx\\
&\qquad\qquad=2\int_{\R^3}\int_{\R^3}f'(|x|)\omega_x\cdot \frac{x-y}{|x-y|^3}|u(t,x)|^2|u(t,y)|^2\,dx\,dy\\
&\qquad\qquad=\int_{\R^3}\int_{\R^3}\frac{\big(f'(|x|)\omega_x-f'(|y|)\omega_y\big)\cdot(x-y)}{|x-y|^3}|u(t,x)|^2|u(t,y)|^2\,dx\,dy
\end{align*}
where in the last line we have just exchanged the role of $x$ and $y$. Inserting in~\eqref{eq:first-Virial}, we arrive at the following expression
\begin{align}
\frac{d}{dt}\pscal{u(t),A_fu(t)}=&4\int_{\R^3}\frac{f'(|x|)}{|x|}|P^\perp_x\nabla u(t,x)|^2\,dx\nonumber\\
&+4\int_{\R^3}f''(|x|)\left|\omega_x\cdot\nabla u(t,x)+\frac{u(t,x)}{|x|}\right|^2\,dx\nonumber\\
&-\int_{\R^3}f^{(4)}(|x|)|u(t,x)|^2\,dx-2Z\int_{\R^3}\frac{f'(|x|)}{|x|^2}|u(t,x)|^2dx\nonumber\\
&+\int_{\R^3}\int_{\R^3}\frac{\big(f'(|x|)\omega_x-f'(|y|)\omega_y\big)\cdot(x-y)}{|x-y|^3}|u(t,x)|^2|u(t,y)|^2\,dx\,dy.
\label{eq:Virial}
\end{align}

For dimensional reasons, it is natural to take $f(|x|)=|x|^3/3$. The following lemma allows to deal with the last potential term in this special case.

\begin{lemma}[Lower bound on the nonlinear term for $f(r)=r^3/3$]\label{lem:estim_potential_x_3}
We have
\begin{equation}
\frac{\big(|x|^2\omega_x-|y|^2\omega_y\big)\cdot(x-y)}{|x-y|^3}\geq \frac{1}{2}
\label{eq:estim_x_3}
\end{equation}
for all $x\neq y\in\R^3$. In the radial case we have
\begin{equation}
\fint_{S^2}\fint_{S^2}\frac{\big(|x|^2\omega_x-|y|^2\omega_y\big)\cdot(x-y)}{|x-y|^3}d\omega_x\,d\omega_y=1
\label{eq:estim_x_3_radial}
\end{equation}
where $\fint_{S^2}d\omega_x=(4\pi)^{-1}\int_{S^2}d\omega_x$ denotes the (normalized) angular integration.
\end{lemma}

\begin{proof}
We compute
\begin{align*}
\frac{(|x|^2\omega_x-|y|^2\omega_y)\cdot(x-y)}{|x-y|^3}=&\frac{r^3+s^3-(r^2s+s^2r)\omega_x\cdot\omega_y}{(r^2+s^2-2rs\omega_x\cdot\omega_y)^{3/2}}\\
=&\frac{1+u^3-(u+u^2)\theta}{(1+u^2-2u\theta)^{3/2}}
\end{align*}
with $x=r\omega_x$, $y=s\omega_y$, $u:=\min(r,s)\max(r,s)^{-1}\in[0,1]$ and $\theta:=\omega_x\cdot\omega_y\in[-1,1]$. Differentiating with respect to $\theta$, we find
$$\frac{\rm d}{{\rm d}\theta}\left(\frac{1+u^3-(u+u^2)\theta}{(1+u^2-2u\theta)^{3/2}}\right)=\frac{u(1+u)(u^2+(1-u)(2-u))-\theta u^2(1+u)}{(1+u^2-2u\theta)^{5/2}}.$$
We have
$u^2+(1-u)(2-u)=u+2(u-1)^2\geq u$
and therefore the numerator is non negative for $u>0$ and $\theta\in[-1,1]$. We conclude that the minimum is attained for $\theta=-1$. The value is
$$\frac{1+u^3+u+u^2}{(1+u)^{3}}=1-\frac{2u}{(1+u)^2}\geq \frac{1}{2}$$
where the minimum is attained for $u=1$. All in all, we find that
$$\frac{(|x|^2\omega_x-|y|^2\omega_y)\cdot(x-y)}{|x-y|^3}\geq \frac12,$$
as was claimed. 
In the radial case we find by explicit integration
\begin{align*}
\fint_{S^2}\fint_{S^2}\frac{\big(|x|^2\omega_x-|y|^2\omega_y\big)\cdot(x-y)}{|x-y|^3} \, d\omega_x\,d\omega_y=&\frac12\int_{-1}^1\frac{1+u^3-(u+u^2)\theta}{(1+u^2-2u\theta)^{3/2}} \, d\theta=1.
\end{align*}
This concludes the proof of Lemma~\ref{lem:estim_potential_x_3}
\end{proof}

\medskip

For $f(r)=r^3/3$, the previous estimates gives 
\begin{multline}
\frac{d}{dt}\pscal{u(t),A_fu(t)}=4\int_{\R^3}|x|\,|P^\perp_x\nabla u(t,x)|^2\,dx+8\int_{\R^3}|x|\left|\omega_x\cdot\nabla u(t,x)+\frac{u(t,x)}{|x|}\right|^2dx\\
+\kappa N^2-2ZN
\label{eq:first-Virial-x3}
\end{multline}
where $\kappa=1$ in the radial case and $\kappa=1/2$ otherwise. If $u$ is a stationary state, then the left side is independent of $t$ and this is a new proof that $N<4Z$ (resp. $2Z$) for bound states. Equation~\eqref{eq:first-Virial-x3} is a new \emph{monotonicity formula} for the Coulombic Hartree equation, when $N\geq 4Z$ (resp. $N\geq2Z$ in the radial case).

\subsection*{Step 2. The Localized Virial Estimate}
We now use a localized virial estimate, which means that we choose a virial function $f_R$ which behaves like $|x|^3/3$ on a ball of radius $R$ and like $|x|$ at infinity. We will take $f_R$ of the form
$$f_R(|x|)=R^{3}f(|x|/R)$$
for 
\begin{equation}
f(r)=r-\arctan r,
\label{eq:choice_f}
\end{equation}
which we have chosen to have
\begin{equation}
f'(r)=\frac{r^2}{1+r^2}=1-\frac{1}{1+r^2}.
\label{eq:choice_f_d}
\end{equation}
Clearly, the first derivative $f'$ is non-decreasing and positive. Hence $x\mapsto f(|x|)$ is a convex function on $\R^3$.
The following lemma gathers some important properties of $f$, which are the `localized' equivalent of Lemma~\ref{lem:estim_potential_x_3} above.

\begin{lemma}[The virial function $f$]\label{lem:f}
Let $f$ be as in~\eqref{eq:choice_f}.
We have
\begin{equation}
\frac{\big(f'(|x|)\omega_x-f'(|y|)\omega_y\big)\cdot(x-y)}{|x-y|^3}\geq \frac{1}{2}\frac{f'(|x|)}{|x|^2}\,\frac{f'(|y|)}{|y|^2}
\label{eq:estim_potential_f}
\end{equation}
for all $x\neq y\in\R^3$. In the radial case, we get
\begin{multline}
\fint_{S^2}\fint_{S^2}\frac{\big(f'(|x|)\omega_x-f'(|y|)\omega_y\big)\cdot(x-y)}{|x-y|^3} \, d\omega_x\,d\omega_y\\
=\frac{f'(\max(|x|,|y|))}{\max(|x|,|y|)^2}=\frac{1}{1+\max(|x|^2,|y|^2)}\geq \frac{f'(|x|)}{|x|^2}\,\frac{f'(|y|)}{|y|^2}.
\label{eq:estim_potential_f_radial}
\end{multline}
\end{lemma}

\begin{proof}
Like in Lemma~\ref{lem:estim_potential_x_3}, we write
\begin{equation}
\frac{\big(f'(|x|)\omega_x-f'(|y|)\omega_y\big)\cdot(x-y)}{|x-y|^3}=\frac{rf'(r)+sf'(s)-\theta(sf'(r)+rf'(s))}{(r^2+s^2-2rs\theta)^{3/2}}
\label{eq:calculation_theta}
\end{equation}
with $r=|x|$, $s=|y|$ and $\theta=\omega_x\cdot\omega_y\in[-1,1]$. Differentiating with respect to $\theta$, we find
$$\frac{r(2r^2-s^2)f'(r)+s(2s^2-r^2)f'(s)-\theta rs(sf'(r)+rf'(s))}{(r^2+s^2-2rs\theta)^{5/2}}.$$
Since $f'>0$, the numerator is positive for $\theta\leq\theta_c$ and negative for $\theta\geq\theta_c$. Regardless whether $\theta_c\in[-1,1]$ or not, the minimum of the function in~\eqref{eq:calculation_theta} is attained at $\theta=\pm1$.
For $\theta=-1$, we find
\begin{equation*}
\frac{f'(r)+f'(s)}{(r+s)^{2}}= \frac{r^2+s^2+2r^2s^2}{(1+r^2)(1+s^2)(r+s)^{2}}.
\end{equation*}
Now we remark that
$$\frac{r^2+s^2+2r^2s^2}{(r+s)^{2}}=\frac12+\frac{(r-s)^2+4r^2s^2}{2(r+s)^{2}}\geq \frac12,$$
and therefore
$$\frac{f'(r)+f'(s)}{(r+s)^{2}}\geq \frac1{2(1+r^2)(1+s^2)}.$$
For $\theta=1$, we find
\begin{equation*}
\frac{|f'(r)-f'(s)|}{(r-s)^{2}}= \frac{|r^2(1+s^2)-s^2(1+r^2)|}{(1+r^2)(1+s^2)(r-s)^{2}}=\frac{r+s}{|r-s|}\frac{f'(r)}{r^2}\frac{f'(s)}{s^2}.
\end{equation*}
We have, with $u=\min(r,s)\max(r,s)^{-1}$,
$$\frac{r+s}{|r-s|}=\frac{1+u}{1-u}=1+\frac{2u}{1-u}\geq 1.$$
We conclude that
$$\frac{\big(f'(|x|)\omega_x-f'(|y|)\omega_y\big)\cdot(x-y)}{|x-y|^3}\geq \frac12\frac{f'(|x|)}{|x|^2}\frac{f'(|y|)}{|y|^2},$$
for all $x\neq y\in\R^3$, as was stated.

In the radial case we have to compute the integral over the angle explicitly. We use the notation $r_<:=\min(r,s)$ and $r_>:=\max(r,s)$, and we get
\begin{align*}
&\fint_{S^2}\fint_{S^2}\frac{\big(f'(|x|)\omega_x-f'(|y|)\omega_y\big)\cdot(x-y)}{|x-y|^3}d\omega_x\,d\omega_y\\
&\qquad\qquad=\frac12\int_{-1}^1\frac{rf'(r)+sf'(s)-\theta(sf'(r)+rf'(s))}{(r^2+s^2-2rs\theta)^{3/2}}\,d\theta\\
&\qquad\qquad=\frac{rf'(r)+sf'(s)}{r_>^3}\frac12 \int_{-1}^1d\theta\frac{1}{(1+u^2-2u\theta)^{3/2}}\\
&\qquad\qquad\qquad\qquad -\frac{\big(rf'(s)+sf'(r)\big)}{r_>^3}\frac12 \int_{-1}^1d\theta\frac{\theta}{(1+u^2-2u\theta)^{3/2}}\\
&\qquad\qquad=\frac{rf'(r)+sf'(s)}{r_>^3}\frac{1}{1-u^2}-\frac{\big(rf'(s)+sf'(r)\big)}{r_>^3}\frac{u}{1-u^2}\\
&\qquad\qquad=\frac{1}{r_>(r_>^2-r_<^2)}\left(r_>f'(r_>)+r_<f'(r_<)-\big(r_<f'(r_>)+r_>f'(r_<)\big)r_</r_>\right)\\
&\qquad\qquad=\frac{1}{r_>(r_>^2-r_<^2)}\left(r_>f'(r_>)-r_<^2f'(r_>)/r_>\right)\\
&\qquad\qquad=\frac{f'(r_>)}{r_>^2}.
\end{align*}
Note also that the previous calculation is valid for an arbitrary radial differentiable function $f$, and not just the specific $f$ chosen above. The proof of Lemma \ref{lem:f} is now complete. \end{proof}

\medskip

We apply~\eqref{eq:Virial} for $f_R=R^3f(\cdot/R)$ with $f$ given by~\eqref{eq:choice_f}. We get the expression
\begin{align}
&\frac{d}{dt}\pscal{u(t),A_{f_R}u(t)}\nonumber\\
&\ =4R\int_{\R^3}\frac{Rf'(|x|/R)}{|x|}|P^\perp_x\nabla u(t,x)|^2\,dx+4R\int_{\R^3}f''(|x|/R)\left|\omega_x\cdot\nabla u(t,x)+\frac{u(t,x)}{|x|}\right|^2\!dx\nonumber\\
&\quad\qquad-\frac{1}{R}\int_{\R^3}f^{(4)}(|x|/R)|u(t,x)|^2\,dx-2Z\int_{\R^3}\frac{R^2f'(|x|/R)}{|x|^2}|u(t,x)|^2dx\nonumber\\
&\quad\qquad+R^2\int_{\R^3}\int_{\R^3}\frac{\big(f'(|x|/R)\omega_x-f'(|y|/R)\omega_y\big)\cdot(x-y)}{|x-y|^3}|u(t,x)|^2|u(t,y)|^2\,dx\,dy.
\label{eq:first-localized-Virial}
\end{align}
We now define the localized mass by
\begin{equation}
M_R(t):=\int_{\R^3}\frac{f'_R(|x|)}{|x|^2}|u(t,x)|^2dx=\int_{\R^3}\frac{1}{1+R^{-2}|x|^2}|u(t,x)|^2dx.
\end{equation}
Using~\eqref{eq:estim_potential_f} (resp.~\eqref{eq:estim_potential_f_radial} in the radial case), we get the lower bound
\begin{equation}
\frac{d}{dt}\pscal{u(t),A_{f_R}u(t)}
\geq-\frac{1}{R}\int_{\R^3}f^{(4)}(|x|/R)|u(t,x)|^2\,dx-2ZM_R(t)+\kappa M_R(t)^2
\label{eq:second-localized-Virial}
\end{equation}
with $\kappa=1$ in the radial case and $\kappa=1/2$ otherwise. Finally we remark that
$$-f^{(4)}(r)=24\, r\, \frac{1 - r^2}{(1 + r^2)^4}\geq -24 \frac{r^3}{(1 + r^2)^4}\1(r\geq1)\geq - \frac{3}{1 + r^2}=-3\frac{f'(r)}{r^2}$$
(the best numerical constant is $1.33$ instead of $3$) and we get our final lower bound
\begin{equation}
\frac{d}{dt}\pscal{u(t),A_{f_R}u(t)}
\geq-\left(2Z+\frac{3}{R}\right)\,M_R(t)+\kappa M_R(t)^2.
\label{eq:final_lower_bound}
\end{equation}

To conclude our proof of~\eqref{eq:estimate_average}, we average~\eqref{eq:final_lower_bound} over a time interval $[0,T]$ and use Jensen's inequality 
$$\frac1{T} \int_0^TM_R(t)^2\,dt\geq \left(\frac1{T} \int_0^TM_R(t)\,dt\right)^2,$$
to get
\begin{multline}
\frac{\pscal{u(T),A_{f_R}u(T)}-\pscal{u(0),A_{f_R}u(0)}}{T}\\
\geq \kappa\left(\frac1{T} \int_0^TM_R(t)\,dt\right)^2-(2Z+3/R)\left(\frac1{T} \int_0^TM_R(t)\,dt\right). 
\end{multline}
Note that
\begin{align*}
|\pscal{u(t),A_{f_R}u(t)}|&=\left|\pscal{u(t),\Big(p\cdot\nabla f_R(|x|)+\nabla f_R(|x|)\cdot p\Big)u(t)}\right|\\
&\leq 2\sqrt{K}\sqrt{N}\norm{f'_R}_{L^\ii}=2\,\sqrt{K}\sqrt{N}\,R^2
\end{align*}
since $\sup_{r\geq0}f'(r)=1$, and where we recall that $K=\sup_{t}\norm{\nabla u(t)}_{L^2}$. In summary, we conclude that
$$\kappa\left(\frac1{T} \int_0^TM_R(t)\,dt\right)^2-(2Z+3/R)\left(\frac1{T} \int_0^TM_R(t)\,dt\right)\leq 4\,\sqrt{KN}\,\frac{R^2}T.$$
Using $\sqrt{1+u}\leq 1+u/2$, this implies
\begin{align*}
\frac1{T} \int_0^TM_R(t)\,dt&\leq \frac{2Z+3/R}{2\kappa}+\frac{2Z+3/R}{2\kappa}\sqrt{1+\frac{16\kappa\sqrt{KN}R^2}{(2Z+3/R)^2T}}\\
&\leq \frac{2Z}{\kappa}+\frac{3}R+\frac{4\,\sqrt{KN}R^2}{(2Z+3/R)T}\leq \frac{2Z}{\kappa}+\frac{3}R+\frac{2\,\sqrt{KN}R^2}{ZT},
\end{align*}
which ends the proof of~\eqref{eq:estimate_average}.

\begin{remark}
Our proof works exactly the same for a more general time average based on a positive function $\mu$ such that $\int_0^\ii\mu=1$ and $\mu'$ is a bounded Borel measure. More precisely, we have the estimate
$$\int_0^\ii\frac{\mu(t/T)}T \, dt\;\int_{\R^3}\frac{|u(t,x)|^2}{1+|x|^2/R^2} \, dx\leq \frac{2Z}{\kappa}+\frac{3}R+\frac{\sqrt{KN}R^2}{ZT}\int_0^\ii|\mu'|.$$
For instance, one could take $\mu(t)=e^{-t}.$ 
\end{remark}

\subsection*{Step 3. Estimate on the Local Kinetic Energy}
We show here that the kinetic energy also has a universal upper bound in average, on any ball of radius $R$. This time, we use a localized virial identity based on the function 
$$g_R(|x|)=R^2g(|x|/R)$$
which behaves like $|x|^2$ on $B_R$ and like $|x|$ at infinity. More precisely, we take 
\begin{equation}
g(r)=r-\log(1+r) 
\label{eq:choice_g}
\end{equation}
which is such that
$$g'(r)=\frac{r}{1+r}=1-\frac{1}{1+r}.$$
Clearly $g'$ is positive and non-decreasing, therefore $x\mapsto g(|x|)$ is convex on $\R^3$.

We use the lower bound~\eqref{eq:radial-commutator-estimate} with $\alpha=0$ and we get, by the same calculations as before,
\begin{align}
\frac{d}{dt}\pscal{u(t),A_{g_R}u(t)}=& 4\int_{\R^3}\left(\frac{Rg'(|x|/R)}{|x|}|P_x^\perp\nabla u(t,x)|^2+g''(|x|/R)|\omega_x\cdot\nabla u(t,x)|^2\right)dx\nonumber\\
&-\frac{1}{R^2}\int_{\R^3}\left(g^{(4)}(|x|/R)+4\frac{R\,g^{(3)}(|x|/R)}{|x|}\right)|u(t,x)|^2\,dx\nonumber\\
&+R\int_{\R^3}\int_{\R^3}\frac{\big(\nabla g(x/R)-\nabla g(y/R)\big)\cdot(x-y)}{|x-y|^3}|u(t,x)|^2|u(t,y)|^2\,dx\,dy\nonumber\\
&-2Z\int_{\R^3}\frac{Rg'(|x|/R)}{|x|^2}|u(t,x)|^2\,dx
\label{eq:lower_bound_g}
\end{align}
We denote by
$$K_R(t):=\int_{\R^3}g''(|x|/R)|\nabla u(t,x)|^2\,dx=\int_{\R^3}\frac{|\nabla u(t,x)|^2}{(1+R^{-1}|x|)^2}\,dx$$
the local kinetic energy. Since $x\mapsto g(|x|)$ is convex, then
$(\nabla g(x)-\nabla g(y))\cdot(x-y)\geq0$
for all $x,y\in\R^3$. Also, we notice that
$$g''(r)=\frac{1}{(1+r)^2}\leq \frac{1}{1+r}=\frac{g'(r)}{r}.$$
Finally, we compute
$$g^{(3)}(r)=-\frac{2}{(1+r)^3}\leq0$$
and
$$g^{(4)}(r)=\frac{6}{(1+r)^4}\leq \frac{6}{1+r^2}=6\,\frac{f'(r)}{r^2}.$$
So we arrive at the estimate
\begin{equation}
\frac{d}{dt}\pscal{u(t),A_{g_R}u(t)}\geq 4K_R(t)-\frac{6}{R^2}M_R(t)-2Z\int_{\R^3}\frac{Rg'(|x|/R)}{|x|^2}|u(t,x)|^2\,dx.
\label{eq:almost_final_g}
\end{equation}
In order to control the negative term, we use again the trick of Hardy:
\begin{align*}
0\leq& \int_{\R^3}g''(|x|/R)\left|\nabla u(t,x)+\alpha\,\omega_x\,u(t,x)\right|^2\,dx\\
=&\int_{\R^3}g''(|x|/R)|\nabla u(t,x)|^2\,dx-\alpha\int_{\R^3}{\rm div}\big(\omega_xg''(|x|/R)\big)|u(t,x)|^2\,dx\\
&+\alpha^2\int_{\R^3}g''(|x|/R)|u(t,x)|^2\,dx\\
=&\int_{\R^3}g''(|x|/R)|\nabla u(t,x)|^2\,dx+\alpha^2\int_{\R^3}g''(|x|/R)|u(t,x)|^2\,dx\\
&-2\alpha\int_{\R^3}\frac{g''(|x|/R)}{|x|}|u(t,x)|^2\,dx-\frac{\alpha}R \int_{\R^3}g^{(3)}(|x|/R)|u(t,x)|^2\,dx.
\end{align*}
Therefore, using that $-g^{(3)}(r)=2(1+r)^{-3}\leq2(1+r^2)^{-1}=2f'(r)r^{-2}$ and that $g''(r)=(1+r)^{-2}\leq (1+r^2)^{-1}=f'(r)r^{-2}$, we find
\begin{equation*}
\int_{\R^3}\frac{g''(|x|/R)}{|x|}|u(t,x)|^2\,dx\leq \frac{1}{2\alpha}K_R(t)+\left(\frac{\alpha}2+\frac1R\right)M_R(t).
\end{equation*}
Coming back to the negative term in~\eqref{eq:lower_bound_g}, we write
\begin{align*}
&\int_{\R^3}\frac{Rg'(|x|/R)}{|x|^2}|u(t,x)|^2\,dx\\
&\qquad = \int_{\R^3}\frac{1}{|x|(1+|x|/R)}|u(t,x)|^2\,dx\\
&\qquad= \int_{\R^3}\frac{1}{|x|(1+|x|/R)^2}|u(t,x)|^2\,dx+\frac1R\int_{\R^3}\frac{1}{(1+|x|/R)^2}|u(t,x)|^2\,dx\\
&\qquad \leq \frac{1}{2\alpha}K_R(t)+\left(\frac{\alpha}2+\frac2R\right)M_R(t).
\end{align*}
Inserting in~\eqref{eq:almost_final_g} gives
\begin{equation}
\frac{d}{dt}\pscal{u(t),A_{g_R}u(t)}\geq \left(4-\frac{Z}{\alpha}\right)K_R(t)-Z\left(\alpha+\frac4R+\frac{6}{R^2}\right)M_R(t).
\label{eq:almost_final_g2}
\end{equation}
Taking $\alpha=Z/2$ leads to 
\begin{equation}
\frac{d}{dt}\pscal{u(t),A_{g_R}u(t)}\geq 2K_R(t)-Z\left(\frac{Z}{2}+\frac4R+\frac{6}{R^2}\right)M_R(t).
\label{eq:almost_final_g3}
\end{equation}
To conclude our proof, we average over $t$ in an interval $[0,T]$ using that
$$|\pscal{u(t),A_{g_R}u(t)}|\leq2\sqrt{K}\sqrt{N}\norm{g'_R}_{L^\ii}=2R\sqrt{K}\sqrt{N}$$
and we get
\begin{equation}
\frac1T\int_0^TK_R(t)\,dt\leq Z\left(\frac{Z}{4}+\frac2R+\frac{3}{R^2}\right)\frac1T\int_0^TM_R(t)\,dt+\frac{2R\sqrt{K}\sqrt{N}}{T},
\end{equation}
which concludes the proof of~\eqref{eq:estimate_average_kinetic}.

\subsection*{Step 4. Estimate on $K$}

We end the proof of Theorem~\ref{thm:Hartree} by estimating the maximal value $K$ of the kinetic energy of $u(t)$, in terms of $\norm{u_0}_{H^1}$, using the conservation of energy.

\begin{lemma}[Kinetic energy estimate]\label{lem:estim_H_1}
We have 
\begin{equation}
\norm{\nabla u(t)}_{L^2(\R^3)}^2\leq {Z^2\norm{u_0}^2_{L^2(\R^3)}+ \norm{\nabla u_0}^2_{L^2(\R^3)}+\frac12\norm{u_0}^3_{L^2(\R^3)}\norm{\nabla u_0}_{L^2(\R^3)}}
\label{eq:estim_H_1}
\end{equation}
for all $t\in\R$.
\end{lemma}

\begin{proof}[Proof of Lemma~\ref{lem:estim_H_1}]
By conservation of energy and mass, we find
\begin{align*}
\cE_Z(u_0)=\cE_Z(u)&\geq \pscal{u,\left(-\frac{\Delta}{2}-\frac{Z}{|x|}\right)u}+\frac12{\norm{\nabla u}^2_{L^2(\R^3)}}\\
&\geq -\frac{Z^2}{2}\int_{\R^3}|u_0|^2+\frac12{\norm{\nabla u}^2_{L^2(\R^3)}},
\end{align*}
since $-{\Delta}/{2}-Z\,|x|^{-1}\geq -Z^2/2$ (hydrogen atom). Next, for $x \in \R^3$ and $u\in H^1(\R^3)$, we note the bound
\begin{equation}
\int_{\R^3}\frac{|u(y)|^2}{|x-y|}\,dy\leq \min_{z\geq0}\left(\frac{z}{2}\int_{\R^3}|u|^2+\frac{1}{2z}\int_{\R^3}|\nabla u|^2\right) =\norm{u}_{L^2(\R^3)}\norm{\nabla u}_{L^2(\R^3)} ,
\end{equation}
which gives us
$$\cE_Z(u_0)\leq \norm{\nabla u_0}_{L^2(\R^3)}^2+\frac12 \norm{u_0}_{L^2(\R^3)}^3\norm{\nabla u_0}_{L^2(\R^3)} .$$
Hence,
$$\norm{\nabla u}^2_{L^2(\R^3)}\leq Z^2\norm{u_0}^2_{L^2(\R^3)}+ 2\norm{\nabla u_0}^2_{L^2(\R^3)}+\norm{u_0}^3_{L^2(\R^3)}\norm{\nabla u_0}_{L^2(\R^3)}.$$
\end{proof}

\bigskip

This concludes the proof of Theorem~\ref{thm:Hartree}. \qed

\section{Extensions: Hartree-Fock and Many-Body Schr\"odinger Theories}\label{sec:extensions}

\subsection{Hartree-Fock Theory}\label{sec:Hartree-Fock}

The Hartree-Fock equations describe the nonlinear evolution of a wave function taking the form of a Slater determinant, i.\,e., 
$$\Psi(t)=\frac{1}{\sqrt{N!}}\sum_{\sigma\in\mathfrak{S}_N}{\rm sgn}(\sigma)\;u_{1}(t,x_{\sigma(1)})\cdots u_N(t,x_{\sigma(N)})$$
where the functions $u_1,...,u_N$ model the states of the $N$ electrons. The physical fact that electrons are fermions is expressed in the Pauli principle given by the orthonormality condition
$$\pscal{u_j,u_k}_{L^2}=\delta_{jk}.$$
The Hartree-Fock equations~\cite{LieSim-77,Chadam-76,BovPraFan-76} form a system of $N$ coupled nonlinear equations similar to~\eqref{eq:Hartree-equation}
\begin{equation}
\begin{cases}
\displaystyle i\frac{\partial}{\partial t}u_j=H_u\,u_j,\\[0.2cm]
H_uv=\displaystyle-\Delta v -Z|x|^{-1}v+\sum_{k=1}^N|u_k|^2\ast|x|^{-1}v-\sum_{k=1}^N\big(\overline{u_j}v\big)\ast|x|^{-1}u_k.
\end{cases}
\label{eq:HF-orbitals} 
\end{equation}
One simple way to write the same equation is to introduce the one-body density matrix
$$\gamma(t):=\sum_{k=1}^N|u_k\rangle\langle u_k|$$
which is the orthogonal projection onto the space spanned by the functions $u_1,...,u_N$. Then~\eqref{eq:HF-orbitals} is equivalent to the so-called von Neumann equation
\begin{equation}
\begin{cases}
\displaystyle i\frac{\partial}{\partial t}\gamma=[H_\gamma,\gamma],\\[0.2cm]
H_\gamma v=\displaystyle-\Delta v-Z|x|^{-1}v+\rho_{\gamma(t)}\ast|x|^{-1}v-\int_{\R^3}\gamma(t,x,y)v(y)\,dy.
\end{cases}
\label{eq:HF} 
\end{equation}
Here $\rho_\gamma(x):=\gamma(x,x)$ is the density associated with the matrix $\gamma$. The time-dependent equation ~\eqref{eq:HF} makes in fact sense for any trace-class operator $\gamma$ such that 
$$0\leq\gamma\leq 1\qquad\text{and}\qquad \tr(\gamma)=N,$$
which corresponds to \emph{generalized Hartree-Fock states}~\cite{BacLieSol-94}. Note that the infinite rank case $\mathrm{rank} \, \gamma = +\infty$ is also allowed here. We refer to \cite{Chadam-76,BovPraFan-76} for the proof of global well-posedness for \eqref{eq:HF} with initial data such that $\tr (1-\Delta) \gamma_0 < +\infty$.  

The following result is the equivalent of Theorem~\ref{thm:Hartree} in the Hartree-Fock case.

\begin{theorem}[Long-time behavior of atoms in Hartree-Fock theory]\label{thm:Hartree-Fock}
Suppose $Z>0$ and let $\gamma_0$ be an arbitrary initial datum such that
$$\tr(1-\Delta)\gamma_0<\ii.$$
Denote by $\gamma(t)$ the unique solution of~\eqref{eq:HF}.  Then we have the following estimate 
\begin{equation}
\frac1{T} \int_0^Tdt\;\int_{\R^3}dx\,\frac{\rho_{\gamma(t)}(x)}{1+|x|^2/R^2}\leq 4Z+1+\frac{3}R+\frac{2\,\sqrt{KN}R^2}{ZT}
\label{eq:estimate_average_HF}
\end{equation}
with
$$N:=\tr(\gamma_0)$$
and
\begin{equation}
K:=\sup_{t\geq0}\tr(-\Delta)\gamma(t)\leq {Z^2N+ 2\tr(-\Delta)\gamma_0+N^3\sqrt{\tr(-\Delta)\gamma_0}}.
\label{eq:unif_K_HF} 
\end{equation}
In particular, we have 
\begin{equation}
{\limsup_{T\to\infty}\frac1{T} \int_0^Tdt\int_{|x|\leq r}dx\,\rho_{\gamma(t)}(x)\leq 4Z+1}
\label{eq:limsup_HF}
\end{equation}
for every $r>0$. Similarly, we have the following estimate on the local kinetic energy
\begin{multline}
\frac1{T} \int_0^T\!dt\int_{\R^3}dx\,\frac{\tau_{\gamma(t)}(x)}{(1+|x|/R)^2}\\
\leq \left(\frac{Z^2}{4}+\frac{2Z}R+\frac{3Z}{R^2}\right)\frac1{T} \int_0^Tdt\;\int_{\R^3}\frac{\rho_{\gamma(t)}(x)}{1+|x|^2/R^2}\,dx+\frac{2R\sqrt{K}\sqrt{N}}{T}
\label{eq:estimate_average_kinetic_HF}
\end{multline}
where $\tau_\gamma(x)=-\sum_{k=1}^3(\partial_k\gamma\partial_k)(x,x)$ is the density of kinetic energy, and therefore
\begin{equation}
{\limsup_{T\to\infty}\frac1{T} \int_0^Tdt\int_{|x|\leq r}dx\,\tau_{\gamma(t)}(x)\leq \frac{Z^2}4(4Z+1)}
\label{eq:limsup_K_HF}
\end{equation}
for every $r>0$.

If the initial datum $\gamma_0$ is radial in the sense that 
$$\gamma_0(\mathscr{R} x,\mathscr{R} y)=\gamma_0(x,y),\qquad\forall x,y\in\R^3,\quad\forall \mathscr{R}\in SO(3),$$
then $\gamma(t)$ is radial for all times and the same estimates~\eqref{eq:estimate_average_HF} and~\eqref{eq:limsup_K_HF} hold true with $4Z+1$ replaced by $2Z+1$.  
\end{theorem}

The proof of Theorem~\ref{thm:Hartree-Fock} is very similar to that of Theorem~\ref{thm:Hartree}, the main new difficulty is the control of the exchange term. Thus we only explain how to deal with it.

\begin{proof}[Sketch of the Proof of Theorem~\ref{thm:Hartree-Fock}]
First, we consider a sufficiently smooth radial function $f=f(|x|)$. (Below we will take $f=f_R$, the same as in the proof of Theorem~\ref{thm:Hartree}.) Differentiating with respect to $t$, we find 
\begin{align}
\frac{d}{dt}\,\tr (A_f\gamma)& =i\tr\big([H_\gamma,A_f]\gamma\big)\nonumber\\
&=-\tr\big([p^2,[p^2,f]]\gamma\big)+i\tr\big([V_\gamma,A_f]\gamma\big)-i\tr\big([X_\gamma,A_f]\gamma\big)\label{eq:commutator_HF}
\end{align}
where $V_\gamma=-Z|x|^{-1}+|x|^{-1}\ast\rho_\gamma$ and $X_\gamma$ is the \emph{exchange term} defined by
$$(X_\gamma u)(x)=\int_{\R^3}\frac{\gamma(x,y)}{|x-y|}u(y)\,dy.$$
Note that 
$$i[V_\gamma,A_f]=-2\nabla f\cdot \nabla V_\gamma$$
is a function (that is, a multiplication operator). Analogous to the Hartree case, we thus obtain
\begin{align*}
i\tr\big([V_\gamma,A_f]\gamma\big)=&-2\int_{\R^3}\rho_\gamma(x)\nabla f(x)\cdot \nabla V_\gamma(x)\,dx\\
=& -2Z\int_{\R^3}\frac{f'(|x|)}{|x|^2}\rho_\gamma(t,x)\,dx\\
&\qquad\qquad+\int_{\R^3}\int_{\R^3}\frac{\big(\nabla f(x)-\nabla f(y)\big)\cdot (x-y)}{|x-y|^3}\rho_\gamma(x)\,\rho_\gamma(y)\,dx\,dy.
\end{align*}
The exchange term is controlled using the following fact.

\begin{lemma}[Exchange term]
Let $\tr (1-\Delta) \gamma < +\infty$ and suppose $f : \R^d \to \R$ satisfies $\nabla f \in L^\infty(\R^d)$. Then we have
\begin{equation}
i\tr\big([X_\gamma,A_f]\gamma\big) = \int_{\R^3}\int_{\R^3}\frac{\big(\nabla f(x)-\nabla f(y)\big)\cdot (x-y)}{|x-y|^3}|\gamma(x,y)|^2\,dx\,dy.
\end{equation}
\end{lemma}

\begin{proof}
The proof is an explicit computation:
\begin{align*}
i\tr\big([X_\gamma,A_f]\gamma\big) =&i\tr\big(\big[X_\gamma,\big(p\cdot (\nabla f)+(\nabla f)\cdot p\big)\big]\gamma\big)\\
=&+\int_{\R^3}\int_{\R^3}X_\gamma(x,y)\nabla_y\cdot(\nabla f)(y)\gamma(y,x)\,dx\,dy\\
&\qquad-\int_{\R^3}\int_{\R^3}\gamma(y,x)\nabla_x\cdot(\nabla f)(x)X_\gamma(x,y)\,dx\,dy\\
&\qquad+\int_{\R^3}\int_{\R^3}X_\gamma(x,y)(\nabla f)(y)\cdot\nabla_y\gamma(y,x)\,dx\,dy\\
&\qquad-\int_{\R^3}\int_{\R^3}\gamma(y,x)(\nabla f)(x)\cdot\nabla_xX_\gamma(x,y)\,dx\,dy.
\end{align*}
Integrating by parts for the first two terms we find
\begin{align*}
i\tr\big([X_\gamma,A_f]\gamma\big)
=&-\int_{\R^3}\int_{\R^3}\gamma(y,x)(\nabla f)(y)\cdot \nabla_y X_\gamma(x,y)\,dx\,dy\\
&\qquad+\int_{\R^3}\int_{\R^3}X_\gamma(x,y)(\nabla f)(x)\cdot\nabla_x\gamma(y,x)\,dx\,dy\\
&\qquad+\int_{\R^3}\int_{\R^3}X_\gamma(x,y)(\nabla f)(y)\cdot\nabla_y\gamma(y,x)\,dx\,dy\\
&\qquad-\int_{\R^3}\int_{\R^3}\gamma(y,x)(\nabla f)(x)\cdot\nabla_xX_\gamma(x,y)\,dx\,dy.
\end{align*}
Now we use that
$$\nabla_y X_\gamma(x,y)=\frac{1}{|x-y|}\nabla_y\gamma(x,y)+\gamma(x,y)\nabla_y\frac{1}{|x-y|}$$
and we exchange $x$ and $y$ in the second and fourth integral. The final result is
\begin{align*}
&i\tr\big([X_\gamma,A_f]\gamma\big)\\
&\qquad = -\int_{\R^3}\int_{\R^3}|\gamma(x,y)|^2\left((\nabla f)(y)\cdot \nabla_y \frac{1}{|x-y|}+(\nabla f)(x)\cdot \nabla_x \frac{1}{|x-y|}\right)\,dx\,dy\\
&\qquad = \int_{\R^3}\int_{\R^3}\frac{\big(\nabla f(x)-\nabla f(y)\big)\cdot (x-y)}{|x-y|^3}|\gamma(x,y)|^2\,dx\,dy.
\end{align*}
\end{proof}

Inserting this in~\eqref{eq:commutator_HF} gives the following value for the derivative of the expectation value of $A_f$:
\begin{multline}
\frac{d}{dt}\,\tr (A_f\gamma) =-\tr\big([p^2,[p^2,f]]\gamma\big)-2Z\int_{\R^3}\frac{f'(|x|)}{|x|^2}\rho_\gamma(t,x)\,dx\\
+\int_{\R^3}\int_{\R^3}\frac{\big(\nabla f(x)-\nabla f(y)\big)\cdot (x-y)}{|x-y|^3}\left(\rho_\gamma(x)\rho_\gamma(y)-|\gamma(x,y)|^2\right)\,dx\,dy.
\label{eq:commutator_HF2}
\end{multline}
Since $f$ is convex, we have the operator bound
$$-[p^2,[p^2,f]]\geq -f^{(4)}(|x|),$$
which gives
$$-\tr\big([p^2,[p^2,f]]\gamma\big)\geq -\tr\big(f^{(4)}\gamma\big)=-\int_{\R^3}f^{(4)}(|x|)\,\rho_{\gamma(t)}(x)\,dx$$
because  of $\gamma\geq0$. Thus we can argue exactly as in the Hartree case. We start by taking $f_R$ given by~\eqref{eq:choice_f} and define the local mass by
$$M_R(t):=\int_{\R^3}\frac{f'(|x|)}{|x|^2}\rho_{\gamma(t)}(x)\,dx.$$
Then we use the bound~\eqref{eq:estim_potential_f}, that is
$$\frac{\big(\nabla f_R(x)-\nabla f_R(y)\big)\cdot (x-y)}{|x-y|^3}\geq \frac12 \frac{R^2f'(|x|/R)}{|x|^2}\frac{R^2f'(|y|/R)}{|y|^2},$$
as well as the fact that $\rho_\gamma(x)\rho_\gamma(y)\geq|\gamma(x,y)|^2$ for a.\,e.~$x,y \in \R^3$ (by the Cauchy--Schwarz inequality and the eigenfunction expansion for $\gamma$.) This gives
\begin{multline*}
\int_{\R^3}\int_{\R^3}\frac{\big(\nabla f_R(x)-\nabla f_R(y)\big)\cdot (x-y)}{|x-y|^3}\left(\rho_\gamma(x)\rho_\gamma(y)-|\gamma(x,y)|^2\right)\,dx\,dy\\
\geq \frac{M_R(t)^2}2-\frac12\tr (h_R\gamma h_R \gamma),
\end{multline*}
with $h_R:=R^2f'(|x|/R)|x|^{-2}$. Since $0\leq\gamma\leq 1$ and $0\leq h_R\leq 1$, we have $h_R\gamma h_R\leq (h_R)^2\leq h_R$ and therefore
$$\tr (h_R\gamma h_R \gamma)\leq \tr (h_R\gamma)=M_R(t).$$
We conclude that
\begin{multline*}
\int_{\R^3}\int_{\R^3}\frac{\big(\nabla f(x)-\nabla f(y)\big)\cdot (x-y)}{|x-y|^3}\left(\rho_\gamma(x)\rho_\gamma(y)-|\gamma(x,y)|^2\right)\,dx\,dy\\
\geq \frac{M_R(t)^2-M_R(t)}2.
\end{multline*}
The additional term is responsible for the change of $4Z$ into $4Z+1$. In the radial case, we use~\eqref{eq:estim_potential_f_radial} instead and we get rid of the factor of $1/2$ on the left-hand side. The rest of the proof is exactly the same as in the Hartree case.
\end{proof}

\subsection{Many-Body Schr\"odinger Equation}\label{sec:many-body}

Our method also applies to the linear many-body Schr\"odinger equation
\begin{equation}
\begin{cases}
\displaystyle i\frac{\partial}{\partial t}\Psi(t)=H(N,Z)\Psi(t),\\[0.2cm]
H(N,Z)=\displaystyle\sum_{j=1}^N\left(-\Delta_{x_j}-\frac{Z}{|x_j|}\right)+\frac12\sum_{1\leq k\neq \ell\leq N}\frac{1}{|x_k-x_\ell|},\\[0.3cm]
\Psi(0)=\Psi_0\in H^1\left((\R^3)^N\right),
\end{cases}
\label{eq:many-body2}
\end{equation}
of which the Hartree and Hartree-Fock models are nonlinear approximations. 

The Hamiltonian $H(N,Z)$ is self-adjoint and bounded from below on $L^2((\R^3)^N)$ with domain $H^2((\R^3)^N)$ and quadratic form domain $H^1((\R^3)^N)$. Of particular interest are its restrictions to the symmetric (a.\,k.\,a.~bosonic) and antisymmetric (a.\,k.\,a.~fermionic) subspaces. These are also self-adjoint operators, denoted respectively by $H_s(N,Z)$ and $H_a(N,Z)$. In either of these two subspaces, the essential spectrum of $H_{a/s}(N,Z)$ is a half line $[\Sigma_{a/s}(N,Z),\ii)$ where 
$$\Sigma_{a/s}(N,Z)=\inf \text{Spec}\big(H_{a/s}(N-1,Z)\big),$$
by the HVZ--Theorem~\cite{ReeSim4,CycFroKirSim-87}. It is known that there are no positive eigenvalues~\cite{FroHer-82}, but there might be embedded eigenvalues in $[\Sigma_{a/s}(N,Z),0]$. There exists a critical number of particles $N^c_{a/s}(Z)$ such that $H_{a/s}(N,Z)$ has no eigenvalues below $\Sigma_{a/s}(N,Z)$ for $N>N^c_{a/s}(Z)$, see~\cite{Ruskai-82,Sigal-82,Sigal-84}. For bosons, it is known that 
$$\lim_{Z\to\ii}\frac{N^c_{s}(Z)}{Z} = \tilde\gamma_c$$
where $\tilde\gamma_c\simeq 1.21 \leq \gamma_c$ is the largest number of electrons that ground states can have in Hartree theory~\cite{BenLie-83b,Baumgartner-84,Solovej-90}. For fermions, it was proved in~\cite{LieSigSimThi-88} that
$$\lim_{Z\to\ii}\frac{N^c_{a}(Z)}{Z} = 1.$$
The best bound valid for all $N$ goes back to Lieb~\cite{Lieb-84} and it holds both for bosons and fermions:
$N^c_{a/s}(Z)<2Z+1$.
For fermions, it was recently improved to  $N^c_{a}(Z)<1.22\,Z+3Z^{1/3}$ by Nam~\cite{Nam-12}.

All the previous authors seem to have only studied when the Hamiltonian $H_{a/s}(N,Z)$ ceases to have eigenvalues below its essential spectrum. The question of the existence of embedded eigenvalues in $[\Sigma_{a/s}(N,Z),0]$ does not seem to have been addressed so far. But this is  a relevant problem in the context of the time-dependent equation. Our method allows us to prove that there are no eigenvalue at all when $N\geq 4Z+1$.

\begin{theorem}[Linear many-body Schr\"odinger equation] \label{thm:many-body}
The Hamiltonian $H(N,Z)$ has no eigenvalue when $N\geq 4Z+1$.
\end{theorem}

Here we do not distinguish between the different particle statistics. Thus our result applies to all of $L^2((\R^3)^N)$ and it deals with all possible symmetries. We, however, conjecture that the largest $N$ such that $H_{a/s}(N,Z)$ can have eigenvalues behaves like $N_{a/s}^c(Z)$ for large $Z$.

\begin{proof}
Let $\Psi\in H^2((\R^3)^N)$ be an eigenfunction of $H(N,Z)$ and let $f_R(|x|)=R^3f(|x|/R)$ be as in~\eqref{eq:choice_f}. Then we write
\begin{align*}
0=&\pscal{\Psi,i\left(H(N,Z)\sum_{j=1}^N(A_{f_R})_{x_j}-\sum_{j=1}^N(A_{f_R})_{x_j}H(N,Z)\right)\Psi}\\
=&\sum_{j=1}^N\pscal{\Psi,i[p_j^2,(A_{f_R})_{x_j}]]\Psi}
-2\sum_{j=1}^N\pscal{\Psi,\nabla f_R(x_j)\cdot\nabla_{x_j}\!\left(-\frac{Z}{|x_j|}+\frac12\sum_{k\neq j}\frac{1}{|x_j-x_k|}\right)\Psi}\\
>&-\frac1R\int_{\R^3}f^{(4)}\left(\frac{|x|}{R}\right)\rho_\Psi(x)\,dx-2Z\int_{\R^3}\frac{R^2f'(|x|/R)}{|x|^2}\rho_\Psi(x)\,dx\\
&\qquad\qquad\qquad +\pscal{\Psi,\left(\sum_{1\leq j\neq k\leq N}\frac{\big(\nabla f_R(x_j)-\nabla f_R(x_k)\big)\cdot(x_j-x_k)}{|x_j-x_k|^3}\right)\Psi}.
\end{align*}
Using~\eqref{eq:estim_potential_f}, we get
\begin{align*}
&\pscal{\Psi,\left(\sum_{1\leq j\neq k\leq N}\frac{\big(\nabla f_R(x_j)-\nabla f_R(x_k)\big)\cdot(x_j-x_k)}{|x_j-x_k|^3}\right)\Psi}\\
&\qquad\qquad\geq \frac12\pscal{\Psi,\left(\sum_{1\leq j\neq k\leq N}\frac{R^2f'_R(|x_j|)}{|x_j|^2}\frac{R^2f'_R(|x_k|)}{|x_k|^2}\right)\Psi}\\
&\qquad\qquad= \frac12\pscal{\Psi,\left(\sum_{j=1}^N\frac{R^2f'_R(|x_j|)}{|x_j|^2}\right)^2\Psi}- \frac12\pscal{\Psi,\left(\sum_{j=1}^N\left(\frac{R^2f'_R(|x_j|)}{|x_j|^2}\right)^2\right)\Psi}\\
&\qquad\qquad\geq \frac12\pscal{\Psi,\left(\sum_{j=1}^N\frac{R^2f'_R(|x_j|)}{|x_j|^2}\right)\Psi}^2- \frac12\pscal{\Psi,\left(\sum_{j=1}^N\frac{R^2f'_R(|x_j|)}{|x_j|^2}\right)\Psi}\\
& \qquad\qquad= \frac12\left(\int_{\R^3}\frac{R^2f'(|x|/R)}{|x|^2}\rho_\Psi(x)\,dx\right)^2- \frac12\int_{\R^3}\frac{R^2f'(|x|/R)}{|x|^2}\rho_\Psi(x)\,dx.
\end{align*}
In the last line we have used Jensen's inequality as well as the fact that $f'(r)/r^2=1/(1+r^2)\leq 1$. 
Passing to the limit as $R\to\ii$ gives $N<4Z+1$.
\end{proof}

Since $H(N,Z)$ has no eigenvalue when $N\geq 4Z+1$, it follows from the known existence of scattering and the asymptotic completeness~\cite{Derezinski-93,SigSof-94,HunSig-00b} that any solution $\Psi(t)$ of the time-dependent equation~\eqref{eq:many-body2} behaves (in an appropriate sense) as a superposition of bound states of $H(k,Z)$ with $k< 4Z+1$ plus a scattering part. In particular, it is possible to prove that 
$$\limsup_{t\to\ii}\int_{|x|\leq r}\rho_{\Psi}(t,x)\,dx\leq 4Z+1.$$
By using argument in the proof of Theorem~\ref{thm:many-body} and following step by step the method of Section~\ref{sec:proof}, one can get a simple proof of the weaker result
$$\limsup_{T\to\ii}\frac1T\int_0^T\,dt\int_{|x|\leq r}\rho_{\Psi}(t,x)\,dx\leq 4Z+1.$$

\bigskip
\noindent\textbf{Acknowledgments.}  E.\,L.~was partially supported through a Steno fellowship from the Danish Research Council (FNU). M.\,L.~acknowledges financial support from the French Ministry of Research (ANR-10-BLAN-0101) and from the European Research Council under the European Community's Seventh Framework Programme (FP7/2007-2013 Grant Agreement MNIQS 258023).


\end{document}